\documentclass[a4paper,11pt,reqno]{article}

\usepackage[margin=3cm,top=2cm]{geometry}
\usepackage{amsmath}
\usepackage{amsthm}
\usepackage{amsfonts}
\usepackage{amssymb}
\usepackage{enumerate}
\usepackage{color}
\usepackage{url}
\usepackage{centernot}
\usepackage{xspace}
\usepackage{hyperref}

\newcommand\ext{{\rm{Ext}}}
\newcommand\supp{{\rm{supp}}}

\newcommand\cds{S}

\newcommand\cG{{\mathcal G}}
\newcommand\cH{{\mathcal H}}

\newcommand\eps{\varepsilon}

\newcommand\N{\mathbb{N}}

\newcommand\Q{\mathbb{Q}}
\newcommand\Z{\mathbb{Z}}
\newcommand\R{\mathbb{R}}

\newcommand\ton{^{(n)}}

\newcommand\tog[1]{^{(#1)}}

\newcommand\wto{{\xrightarrow{w}}}

\newcommand\fd{Fr\'echet differentiable\xspace}
\newcommand\fr{Fr\'echet }

\renewcommand{\le}{\leqslant}
\renewcommand{\ge}{\geqslant}

\title{Random Geometric Graphs in Reflexive Banach Spaces}
\author{J\'ozsef Balogh\footnote{
Department of Mathematics, University of Illinois at Urbana-Champaign, IL, 
USA. Email: \texttt{jobal@illinois.edu}.
Balogh was supported in part by NSF grants DMS-1764123 and RTG DMS-1937241, FRG DMS-2152488, the Arnold O. Beckman Research Award (UIUC Campus Research Board RB 24012).}
\and 
    Mark Walters\footnote{School of Mathematical Sciences, Queen Mary,
    University of London, London E1 4NS, England. e-mail: \tt{m.walters@qmul.ac.uk.}}
  \and Andr\'as Zs\'ak \footnote{Peterhouse, Cambridge CB2 1RD and Department of Pure Mathematics and Mathematical Statistics, Centre for Mathematical Sciences, University of Cambridge, Wilberforce Road, Cambridge CB3 0WB, United Kingdom}}
\date\today

\begin{document}
\maketitle
\newtheorem{theorem}{Theorem}
\newtheorem{lemma}[theorem]{Lemma}
\newtheorem{proposition}[theorem]{Proposition}
\newtheorem{corollary}[theorem]{Corollary}
\newtheorem*{defn}{Definition}
\newtheorem*{theorem*}{Theorem}
\newtheorem*{lemma*}{Lemma}
\newtheorem*{conjecture}{Conjecture}
\newtheorem{question}{Question}
\newtheorem*{quotedresult}{Lemma A}
\theoremstyle{remark}
\newtheorem*{remark}{Remark}

\begin{abstract}
 We investigate a random
geometric graph model introduced by Bonato and Janssen. The vertices are the points of a countable dense
set $S$ in a (necessarily separable) normed vector space $X$, and each pair of 
points are joined independently with some fixed probability $p$ (with $0<p<1$) if
they are less than distance $1$ apart.
 A countable dense set $S$ in a normed space is
 \emph{Rado}, if the resulting graph is almost surely
unique up to isomorphism: that is any two such graphs are, almost
surely, isomorphic.

Not surprisingly, understanding which sets are Rado is closely related to the geometry of the underlying normed space. It turns out that a key question is in which spaces must step-isometries (maps that preserve the integer parts of distances) on dense subsets  necessarily be isometries. We answer this question for a large class of Banach spaces including all strictly convex reflexive spaces. In the process we prove results on the interplay between the norm topology and weak topology that may be of independent interest.

As a consequence of these Banach space results we show that almost all countable dense sets in strictly convex reflexive spaces are strongly non-Rado (that is, any two graphs are almost surely non-isomorphic). However, we show that there do exist Rado sets even in $\ell_2$. Finally we construct a Banach spaces in which all countable dense set are strongly non-Rado.

\end{abstract}

\section{Introduction}
There have been several papers recently discussing a novel random
geometric graph model introduced by Bonato and Janssen
in~\cite{MR2854782}. The vertices are the points of a countable dense
set $S$ in a (necessarily separable) normed vector space $X$, and any
two points are joined independently with some fixed probability $p$ (with $0<p<1$) if
they are less than distance 1 apart. We remark that the resulting
graph is not only infinite, but also, almost surely, every vertex has
infinite degree. We call the resulting graph $G=G(S,p)$ and we write
$\cG(S,p)$ for the probability space of such graphs.

The simplest case is just $X=\R$ and in that case, Bonato and
Janssen~\cite{MR2854782} proved that, for almost all sets $S$ there is
the surprising property that the resulting graph is almost surely
unique up to isomorphism: that is any two such graphs are, almost
surely, isomorphic. (In fact they also showed the resulting graph is independent of
$p$.) We call a countable dense set in any normed space which has this
property \emph{Rado}.
 
They also generalised this result to show that, in any finite
dimensional space with the sup norm, (i.e., $\ell_\infty^n$), almost
all countable dense sets are Rado. In contrast, again
in~\cite{MR2854782}, they showed that in the Euclidean plane no
countable dense set is Rado: indeed, they showed that, for any such
set, any two random graphs are almost surely non-isomorphic. We say
that such a set is \emph{strongly non-Rado}.

The question now becomes to classify spaces that behave like $\ell_\infty^n$ or like the Euclidean plane.
\begin{question}
  In which spaces $X$ are almost all (or all) sets Rado? In which
  spaces $X$ are almost all (or all) sets strongly non-Rado?
\end{question}

We briefly mention the use of the term `almost all' with regard to
countable dense sets -- there are many possible probability measures
(see e.g., the discussion at the end of Section~2 of~\cite{MR3841851})
but in this paper the choice will not be important -- when we say
`almost all' it will apply in all sensible measures.  For example, the
precise result that Bonato and Janssen proved for sets in $\R$ is that
any countable dense set in $\R$ with no two points at integer
distances is Rado, and it is clear that this constraint holds for
almost all countable dense sets in any reasonable measure.

A key idea of Bonato and Janssen's proofs was relating the question to the
following question in normed space theory. As the related question is
of interest itself we discuss this next. First we need a
definition that will be fundamental throughout the paper.
\begin{defn}(Bonato and Janssen~\cite{MR2854782})
  Let $S$ be a dense set in a normed space $X$ and $T:S\to S$ be any
  bijective mapping. $T$ is a \emph{step-isometry} if, for all $x,y\in
  S$, we have $\lfloor\|x-y\|\rfloor =\lfloor\|T(x)-T(y)\|\rfloor$.

  Equivalently $T$ is a step-isometry if, for any $k\in \N$ we
have $\|x-y\|<k$ if and only if $\|T(x)-T(y)\|<k$.
\end{defn}

The following result shows why step-isometries arise.

\begin{lemma}\label{l:isom=step-iso}(Bonato and Janssen~\cite{MR2854782})
  Suppose that $S$ is a countable dense set in a normed space. Then there is a set $\cH$
  consisting of almost all graphs in $\cG=\cG(S,p)$ such that if
  $G_1,G_2\in \cH$ and $T:S\to S$ induces an isomorphism from $G_1$ to $G_2$, then $T$ is a
  step-isometry on $S$.
 \end{lemma}
 
The key observation in proving this result is that almost all graphs
have the property that, for any two points $x,y\in S$ and $k\in \N$
with $k\ge 2$, the graph distance satisfies $d_G(x,y)\le k$ if and
only if the norm distance satisfies $\|x-y\|<k$. It is immediate that any isomorphism $T$ between two such graphs has the property that 
for any $k\ge 2$ we
have $\|x-y\|<k$ if and only if $\|T(x)-T(y)\|<k$ -- so almost a step-isometry; to complete the proof we would just need to show that $\|x-y\|<1$ if and only if $\|T(x)-T(y)\|<1$. 

In fact, Lemma~\ref{l:isom=step-iso} is stated somewhat differently in~\cite{MR2854782} and this final case does not actually hold in as much generality as stated there. However, it does hold in all separable reflexive spaces (see Lemma~\ref{l:m-step-isometry-continuous} and Corollary~\ref{c:step-isom-cts}).
Regardless, this detail is unimportant -- even though $T$ might not be a step-isometry in $\|\cdot\|$, it \emph{is} a step-isometry in the norm $\|\cdot\|'$ given by $\|x\|'=\|x\|/2$ which is an isometric renorming of $X$.  

A simple explicit consequence is the following lemma.

\begin{lemma}\label{l:countable=not-isomorphic}
  Let $X$ be a normed space and $S$ be a countable dense set with only
  countably many step-isometries on $S$. Then $S$ is strongly
  non-Rado.
\end{lemma}
We remark that the converse is not true (for example, the conclusion to the lemma would hold if some infinite subset $S'$ of $S$ has only countably many step-isometries, and any step-isometry of $S$ maps $S'$ to $S'$).

In some normed spaces, such as $\R$ or $\ell_\infty^d \color{black}$,
there are many step-isometries that are not isometries. For example,
given any increasing bijection $g:[0,1]\to[0,1]$, it is easy to check
that $f(x)=\lfloor x\rfloor+g(x-\lfloor x\rfloor)$ is a step-isometry
on $\R$. However, in some other normed spaces such as $d$-dimensional
Euclidean space ($d\ge 2$) it turns out that all step-isometries on
countable dense subsets are actually isometries.  This raises the
following question.
\begin{question}\label{q:step-isom=isom}
  Which spaces $X$ are such that almost all (or all) countable dense subsets $S$
  have the property that any step-isometry on $S$ is actually an isometry?
\end{question}
Any such isometry on $S$ extends to a bijective isometry on the entire space $X$, which is affine by the Mazur-Ulam Theorem (see e.g., \cite{MR1957004}).
It follows that, in the case where a
space $X$ with this property is finite-dimensional, there are only
countably many isometries of $S$.  So by
Lemma~\ref{l:countable=not-isomorphic}, most (or all) countable dense
subsets of $X$ are strongly non-Rado.

Balister, Bollob\'as, Gunderson, Leader and Walters~\cite{MR3841851}
classified the situation in all finite-dimensional normed spaces. They
showed that in all finite-dimensional normed spaces other than
$\ell_\infty^n$ almost all countable dense sets are strongly non-Rado,
and that in many such spaces, including all strictly convex spaces,
\emph{all} countable dense sets are strongly non-Rado. In fact, they did this by proving the corresponding statements for Question~\ref{q:step-isom=isom}; for example showing that any step-isometry on any countable dense set in a strictly convex finite-dimensional space is actually an isometry.
However, they did
not prove anything in the infinite-dimensional setting and, indeed,
asked that as an explicit open question.

More recently, Bonato, Janssen and Quas~\cite{MR3895055,MR4244349}
gave three examples of infinite-dimensional Banach spaces that have
the property that almost all countable dense sets are Rado; the
examples are $c_0$, $c$ and $C([0,1])$ (the space of sequences tending
to zero with the sup norm, the space of sequences tending to a limit
with the sup norm, and the space of continuous functions on $[0,1]$
with the sup norm, respectively).

Finally, in 2021 \u{S}emrl~\cite{MR4222374} proved that in $\ell_p$
for $1<p<\infty$, any step-isometry on any countable dense subset is
actually an isometry. As we shall see this does not show that all
countable dense subsets in $\ell_p$ are strongly non-Rado (as there are countable
dense subsets of $\ell_p$ with uncountably many isometries). However, it does
show that \emph{almost all }sets in $\ell_p$ are strongly non-Rado (e.g., any
countable dense subset with no repeated distance will have no
non-trivial isometries).

The question is still open for all other Banach spaces --
in~\cite{MR3895055} Bonato, Janssen and Quas specifically ask about
the situation in $\ell_p$ and $L_p$; \u{S}emrl's~\cite{MR4222374}
results answer that for most countable dense sets in
$\ell_p$. However, he does not answer it for $L_p$. Additionally he
explicitly asks which spaces have the property that step-isometries on
dense subsets are isometries.

In this paper we answer these questions for a large class of spaces by
proving the following theorem and corollary.
\begin{theorem}\label{t:strict-convex-step-iso-is-iso}
  Let $X$ be a strictly convex reflexive Banach space of dimension at least $2$, $S$ be a
   dense set, and $T$ be a step-isometry on $S$. Then $T$ is an isometry.
\end{theorem}
\begin{corollary}\label{c:strict-convex}
  Let $X$ be a strictly convex reflexive Banach space of dimension at
  least $2$, and $S$ be a countable dense set with no repeated distance,
  then $S$ is strongly non-Rado.
\end{corollary}
Note that in the graph model (e.g., Lemmas~\ref{l:isom=step-iso} and~\ref{l:countable=not-isomorphic} and Corollary~\ref{c:strict-convex}) it is necessary that $S$ is countable. However, Theorem~\ref{t:strict-convex-step-iso-is-iso}, which is a purely Banach space result, holds without the separability assumption on $X$. 

Corollary~\ref{c:strict-convex} follows
immediately from Theorem~\ref{t:strict-convex-step-iso-is-iso} and
Lemma~\ref{l:countable=not-isomorphic}, since on any set $S$ with no repeated distance, the identity is the only isometry.  The condition that $S$ has no
repeated distance is true for any reasonable definition of `almost
all' point sets.  Note, we actually prove a somewhat stronger result
that includes the above results as special cases: see
Section~\ref{s:step=isom} for details.

However, in contrast to the situation in finite dimensions, in
infinite dimensions a countably dense set can have many isometries
and, in fact, there can be enough isometries for the set to be Rado.
Indeed, we prove that in many strictly convex reflexive Banach spaces,
including the Hilbert space, there do exist Rado sets. More precisely
we prove:
\begin{theorem}\label{t:exist-rado}
  Let $X$ be any space with a 1-symmetric basis (this includes all the spaces 
  $\ell_p$  with $1< p<\infty$). Then there exists a countable
  dense set that is Rado.
\end{theorem}
\noindent%
Note that Theorem~\ref{t:exist-rado} also holds in both $c_0$ where,
as mentioned above, Bonato, Janssen and Quas~\cite{MR3895055} showed that almost all
countable dense sets are Rado, and in $\ell_1$ where the general
situation is still open.

Finally, in contrast we prove:
\begin{theorem}\label{t:exist-no-rado}
  There exists an infinite-dimensional Banach space in which
  \emph{all} countable dense sets are strongly non-Rado.
\end{theorem}

The majority of this paper is focused on proving
Theorem~\ref{t:strict-convex-step-iso-is-iso}.  There are two key steps.
\begin{enumerate}
  \item Prove that every step-isometry on $\cds$ extends (uniquely) to a
    step-isometry on the whole space $X$.
  \item Prove that every step-isometry on the whole space $X$ is actually an isometry.
\end{enumerate}

The first step was proved by Balister, Bollob\'as, Gunderson, Leader
and Walters~\cite{MR3841851} for all finite-dimensional normed
spaces. This was one of the simpler steps in that paper, but the proof
relies on the compactness of the unit ball in finite dimensions. The
unit ball in infinite dimensions is never compact in the norm topology
so we cannot do the same here. In Section~\ref{s:not-extend} we give
an example of a step-isometry on a countable dense set in $c_0$ that
does not extend to a step-isometry on all of $c_0$. This shows that we
need some condition on our space to ensure the existence of such an
extension, and a natural choice is to insist that the space is
reflexive. Indeed, in any reflexive space the unit ball is compact in
the weak topology (which is the same as the weak-* topology in a
reflexive space) and we are able to use this to give the desired
extension. Since the step-isometry condition is very clearly
norm-based, and not weak-topology based, the argument is much more complex
than in finite dimensions.

The second step is not true in general (even in finite
dimensions). Showing exactly when it holds formed the most
substantial part of~\cite{MR3841851}. However, when the norm is strictly convex,
the situation is much simpler than the general case and the arguments
from~\cite{MR3841851} do carry over to infinite dimensions.

\subsection*{Outline of the Paper}
We start by outlining the notation we will use in this paper. In
Section~\ref{s:not-extend} we show that in $c_0$ there exists a
countable dense set $S$ and a step-isometry $T$ on $S$ that does not
extend to a step-isometry on the whole space.

The non-extendable step-isometry can be naturally viewed as a
step-isometry on $c_{00}$, the subspace of $c_0$ consisting of those
vectors with finitely many non-zero terms, and one observation is that
this step-isometry on $c_{00}$ is not continuous. This motivates
Section~\ref{s:continuous}, where we prove that in many Banach spaces
all step-isometries are continuous and indeed we do not know of any
Banach space (i.e., \emph{complete} normed space), where this is not the case.

In Section~\ref{s:extend-reflexive} we show that in reflexive spaces
step-isometries on countable dense subsets can always be extended to
step-isometries on the whole space.  Then in Section~\ref{s:step=isom}
we show that, provided the norm is strictly convex, step-isometries
are in fact isometries and, thus, by combining with the results  in
Section~\ref{s:extend-reflexive}, prove Theorem~\ref{t:strict-convex-step-iso-is-iso}.

Then in Section~\ref{s:ellp-rado-set} we prove
Theorem~\ref{t:exist-rado} showing that there exist Rado sets in
spaces with a 1-symmetric basis.  Finally in Section~\ref{s:no-rado} we prove
Theorem~\ref{t:exist-no-rado}.

We conclude with some open questions.

\section{Notation and Preliminaries}

As many Banach spaces are naturally defined as sequence spaces
(e.g. $\ell_2$), and as we will want to consider dual spaces, we do
not use the same notation as in previous papers.

Our Banach space will generally be called $X$; vectors in $X$ will be
$x,y,z,\ldots$ and where $X$ is a sequence space $x_i$ will be the
$i^{\textrm{th}}$ coordinate. Sequences of vectors in $X$ will be
denoted $x\ton$. Elements of the dual space $X^*$ will be denoted by
$f$.  The countable dense set will be $S$ and the step-isometry from $S$
to $S$ or from $X$ to $X$ will be called $T$.

One useful result is that step-isometries on the whole space
necessarily preserve integer distances. We state the result in slightly more generality replacing step-isometry by a slightly weaker notion that often arises naturally. We say a bijective map $T:S\to S$ is a \emph{step-isometry for large distances}  if there exists $m_0$ such that for all $m\ge m_0$ we have
$\|x-y\|<m$ if and only if $\|T(x)-T(y)\|<m$.

\begin{lemma}\label{l:integer-distances}
  Let $S$ be a dense subset of a normed space $X$ and $T:S\to S$ be a step-isometry for large distances. Then, for all $x,y\in S$ and $k\in \N$ we have
  $\|x-y\|\le k$ if and only if $\|T(x)-T(y)\|\le k$.

  In particular, if $T$ is a step-isometry then it preserves integer
  distances: i.e., $\|x-y\|= k$ if and only if $\|T(x)-T(y)\|= k$.
\end{lemma}

We remark that the converse does not hold: the map $T$ defined in Section~\ref{s:not-extend} viewed as a map on $\ell_\infty$ has the property that, for all $x,y\in S$ and $k\in \N$ we have
  $\|x-y\|\le k$ if and only if $\|T(x)-T(y)\|\le k$ but it is not a step-isometry for large distances and does not preserve integer distances. We do not know if the converse always holds on reflexive spaces.

\begin{proof}
   Suppose $x,y\in S$ with $\|x-y\|>k$ for some $k\in \N$. Let $m_0$ be as in the definition of step-isometry for large distances. Pick $u,v\in
  X$ collinear with $x,y$ such that $\|u-x\|< m_0$, $\|v-y\|< m_0$ and
  $\|u-v\| > k+2m_0$. Providing we pick $u',v'\in S$ sufficiently
  close to $u$ and $v$ then we will have $\|u'-x\|< m_0$, $\|v'-y\|<
  m_0$ and $\|u'-v'\| > k+2m_0$.

  By the hypothesis on $T$, we have $\|T(u')-T(x)\|< m_0$,
  $\|T(v')-T(y)\|< m_0$ and $\| T(u')- T(v')\|\ge k+2m_0$. Thus, by
  triangle inequality, $\|T(x)-T(y)\|> k$. The converse holds
  by applying the same argument to $T^{-1}$. 
\end{proof}

In fact, in many cases, any step-isometry for large distances is automatically a step-isometry.  Indeed, we
do not know of any case where it is not.
Whilst we cannot prove that the two notions coincide in general, they are the same whenever $S$ is the whole normed space. It is also the case that if $T$ is a step-isometry for large distances such that both $T$ and $T^{-1}$ are continuous then $T$ is a step-isometry. As we shall see in the next section
in many spaces all step-isometries for large distances are continuous.

\begin{lemma}\label{l:m-step-isometry-vector-space}
  Let $X$ be a normed space and $T:X\to X$ be a step-isometry for large distances.  Then $T$ is a
  step-isometry.
\end{lemma}
\begin{proof}
  Suppose $x,y\in X$ and $\|x-y\|\ge k$ for some $k\in \N$. Let $m_0$ be as in the definition of step-isometry for large distances.  Pick
  $u,v\in X$ collinear with $x,y$ such that $\|u-x\|=m_0$,
  $\|v-y\|=m_0$ and $\|u-v\|\ge k+2m_0$. Then, by
  Lemma~\ref{l:integer-distances}, we have $\|T(u)- T(x)\|\le m_0$ and
  $\|T(v)-T(y)\|\le m_0$. Also, by the hypothesis on $T$, we have
  $\|T(u)-T(v)\|\ge k+2m_0$. Thus, by triangle inequality,
  $\|T(x)-T(y)\|\ge k$ as required. Again the converse holds by
  applying the argument to $T^{-1}$.
\end{proof}
\begin{lemma}\label{l:m-step-isometry-continuous}
  Let $S$ be a dense subset of a normed space $X$ and $T:S\to S$ be a step-isometry for large distances. Further suppose that $T$ and $T^{-1}$ are both
  continuous. Then $T$ is a step-isometry.
\end{lemma}
\begin{proof}
  By Lemma~\ref{l:integer-distances} we know that for all $x,y\in S$
  and $k\in \N$ we have $\|x-y\|\le k$ if and only if
  $\|T(x)-T(y)\|\le k$. We need to show that the same holds if we
  replace both inequalities by strict inequalities. It suffices to
  show that $\|x-y\|=k$ if and only if $\|T(x)-T(y)\|= k$

  Suppose $\|x-y\|=k$. Let $x\ton$ be a sequence in $S$ converging to
  $x$ with the property that $\|x\ton-y\|>k$ for all $n$. By
  Lemma~\ref{l:integer-distances} we see that
  $\|T(x\ton)-T(y)\|>k$. Since $T$ is continuous, $T(x\ton)\to
  T(x)$, and so we have $\|T(x)-T(y)\|\ge k$. Since
  Lemma~\ref{l:integer-distances} shows that $\|T(x)-T(y)\|\le k$, we have
  $\|T(x)-T(y)\|=k$ as required. Again the converse holds by applying
  the same argument to $T^{-1}$.
\end{proof}

\section{Step-isometries may not extend.}\label{s:not-extend}
We start by showing that, unlike in the finite-dimensional setting,
there exist step-isometries on countable dense sets that cannot be
extended to step-isometries on the whole space.
\begin{theorem}\label{t:step-isom-not-extend}
  There exists a countable dense set $S$ in $c_0$, and a step-isometry
  $T$ mapping $S$ to itself that does not extend to a step-isometry on
  the whole of $c_0$.
\end{theorem}
\begin{proof}
  We define $S$ to be the set of all rational sequences in $c_{00}$
  (the space of eventually zero sequences). Obviously this is
  countable and dense in $c_0$.

We define $T:S\to S$ coordinate-wise. Define $g_n:[0,1]\to [0,1]$ to be
the piecewise linear function with
\[
  g_n(0)=0,\qquad g_n\left(\tfrac{1}{n+1}\right)=1-\tfrac{1}{n+1}, \qquad
  g_n(1)=1.
\]
Define $h_n:\R\to \R$ by
$h_n(x)=\lfloor x \rfloor+g_n(x-\lfloor x \rfloor)$, and note that
$h_n$ is a step-isometry and that it is a bijection $\Q$ to $\Q$.  Finally
define $T$ by
\[
T\left((x_n)_{n=1}^\infty\right)=\left(h_n(x_n)\right)_{n=1}^\infty.
\]
It is clear that $T$ is a bijection of $S$ to itself and, since $T$ is
a step-isometry in each coordinate, it is a step-isometry on $S$
(which follows immediately since $S$ is a subset of $c_{00}$).

We remark that $T$ is \emph{not} continuous. Indeed, the sequence
$\tfrac{1}{n+1}e_n$ (where $e_n$ are the standard basis vectors) tends
to zero, but
$\|T(\tfrac{1}{n+1}e_n)\|=\|(1-\tfrac{1}{n+1}) e_n\|=1-\tfrac{1}{n+1}\ge\frac12$
for all $n$.

Suppose, for a contradiction, that $T$ does extend to a step-isometry
$\bar T$ on the whole of $c_0$. Let $x\in c_0$ be the point
$x=(\tfrac{1}{n+1})_{n=1}^\infty$. We consider the possible values for
$\bar T(x)$. It is immediate that for any $\eps$ with $0<\eps<1$,
$y_\eps=(1+\tfrac{1}{n+1}-\eps)e_n$ is less than distance 1 from $x$
as is $z_\eps=(-1+\tfrac{1}{n+1}+\eps)e_n$. Hence, since $\bar T$ is a
step-isometry, $\bar T(x)$ has to be in
$B^\circ(\bar T(y_\eps),1)\cap B^\circ(\bar T(z_\eps),1)$, and hence
\[\bar T(x)\in \bigcap_{\eps>0}B^\circ(\bar T(y_\eps),1)\cap B^\circ(\bar T(z_\eps),1).\]
However, if $\eps\in\Q$ then $y_\eps$ and $z_\eps$ are both in $S$ so
$\bar T(y_\eps)=T(y_\eps)$ and $\bar T(z_\eps)=T(z_\eps)$. In particular,
\[\bar T(x)\in
\bigcap_{\substack{\eps>0\\\eps\in\Q}}B^\circ(T(y_\eps),1)\cap
  B^\circ(T(z_\eps),1)\tag{*}.\] Provided $\eps<\tfrac{1}{n+1}$ we have
$T(y_\eps)=\left(2-\tfrac{1}{n+1}-n\eps\right)e_n$ and
$T(z_\eps)=\left(-\tfrac{1}{n+1}+\eps/n\right)e_n$.  It is easy to see that
this implies that any point in the intersection $(*)$ must have its $n^\textrm{th}$
coordinate equal to $1-\tfrac{1}{n+1}$. Thus $\bar T (x)$ must have
$n^\textrm{th}$ coordinate $1-\tfrac{1}{n+1}$.

This is true for all $n$, so $\bar T(x)=(1-\tfrac{1}{n+1})_{n=1}^\infty$, but this is
not in $c_0$ which is the desired contradiction.
\end{proof}

We conclude this section with two observations about the proof
above. First we note that the same definition of $T$ naturally gives a
step-isometry on the whole of $c_{00}$. In particular, there are
step-isometries on the normed space $c_{00}$ that are not continuous.

Secondly, the map $T$ does extend very naturally to a map $\bar
T:\ell_\infty\to \ell_\infty$: just use the coordinate-wise
definition. Indeed, that is what we showed had to happen for the
specific point $x=(\tfrac{1}{n+1})_{n=1}^\infty$. However, this map is
not a step-isometry -- indeed, for $x=(\tfrac{1}{n+1})_{n=1}^\infty$
as in the proof above, we have $\|x\|=1/2$ but $\|\bar T(x)\|=1$
which, since $T(0)=0$, breaks the step-isometry
condition. Alternatively, as we shall see in the next section, all
step-isometries on $\ell_\infty$ are continuous so this
(non-continuous) map cannot be a step-isometry.

\section{Step-isometries on many spaces are continuous}\label{s:continuous}
In the previous section we have seen an example of a step-isometry on
a normed space ($c_{00}$ with the sup norm) that is not continuous. In this section we
show that in many Banach spaces \emph{all} step-isometries are
continuous and, indeed, we do not know of any \emph{Banach} space where this
is not the case. We remark that this property is not preserved by isomorphism, only by isometry; later in this section we will give an example of a normed space isomorphic to $c_{00}$ with the sup norm on which all step-isometries are continuous.

By the definition of a step-isometry we know that the pre-image of any
open ball of integer radius is an open ball of (the same) integer
radius. In many Banach spaces, the open unit balls form a subbase for
the (norm) topology and hence, for these spaces, all step-isometries
are continuous. Since the set of unit balls is closed under
translation, in order to show that the unit balls are a subbase we
just need to show that, for any $\eps$ there is an intersection of
balls with diameter less than $\eps$. This is not the case in all
Banach spaces: indeed, it is easy to see that in $c_0$ any (non-empty)
finite intersection of unit open balls has diameter $2$. However, if the closed unit ball has a
`corner' then the open unit balls \emph{do} form a subbase.

There are many definitions of `corner' from `extreme point' to
`strongly exposed point' (see Bourgin~\cite{MR0704815}). We will
use one of the weaker definitions, namely: `strongly extreme point'. A point $x$ in the unit ball is a
\emph{strongly extreme point} if for all $\eps>0$ there exists
$\delta>0$ such that for any $y$ with $\|x+y\|<1+\delta$ and
$\|x-y\|<1+\delta$ we have $\|y\|<\eps$.

\begin{lemma}
  Let $X$ be a Banach space for which the unit ball B has a strongly
  extreme point $x$. Then the open unit balls form a subbase for the
  norm topology on $X$.
\end{lemma}
\begin{proof}
  As mentioned above we just need to find finite intersections of open
  balls of arbitrarily small diameter. Let $x$ be a strongly extreme
  point of the unit ball. Then given $\eps>0$, let $\delta$ be as in
  the definition of strongly extreme point for $x$. Consider the two
  balls $B^\circ(0,1+\delta)$ and $B^\circ(2x,1+\delta)$. Obviously
  $x$ is in the intersection. Suppose $x+y$ is in the
  intersection. Then $\|x+y\|< 1+\delta$ and, since $2x-(x+y)=x-y$, we
  also have $\|x-y\|<1+\delta$. Therefore, by the definition of a
  strongly extreme point, $\|y\|<\eps$. In particular the intersection
  is a subset of $B(x,\eps)$.

  Hence, scaling everything by $1/(1+\delta)$ we see that the
  intersection of the two open unit balls $B^\circ(0,1)$ and
  $B^\circ(\frac{2}{1+\delta}x,1)$ has diameter at most $\eps$ and
  this completes the proof.
\end{proof}

Now that we have this lemma we just need to show that the unit balls
in `many' spaces have strongly extreme points. The unit ball in $c_0$
has no extreme points so it definitely has no strongly extreme point
and, indeed, we have seen that the unit balls do not form a subbase in this
case.

However, it is known that the unit ball in any space with the
Radon-Nikodym property does have a strongly extreme point (and more, see Phelps~\cite{MR0352941}).  All separable dual spaces 
have the Radon-Nikodym property, and it follows that, in any separable
dual space and, in particular, in any separable reflexive space, all
step-isometries are norm continuous.

We remark that the
question of whether every step-isometry is continuous makes sense in
any Banach space, separable or not, even though our original
motivation only applies in separable spaces. Indeed, the class of spaces with the Radon-Nikodym property is rather
large and includes many non-separable spaces, including \emph{all} reflexive spaces and, more generally, the dual of any Asplund space. 

Whilst the Radon-Nikodym property is a sufficient condition for the
unit balls to be a subbase, it is not a necessary condition. (In fact, Hu~\cite{MR1152279}) showed that a space $X$ has the Radon-Nikodym property if and only if the unit ball in any equivalent norm has a strongly extreme point.)  Indeed,
$\ell_\infty$ does not have the Radon-Nikodym property, but it is easy
to see that its unit ball does have a strongly extreme point (e.g.,
the constant 1 sequence), so all step-isometries on it are continuous.

Similarly, on any space of continuous functions of a compact set with the
sup norm (i.e., a space $C(K)$) the identically 1 function is a strongly extreme point,
and again any step-isometry on such a space is necessarily
continuous. This includes the space $c$. 

In exactly the same way we see that all step-isometries on the space of eventually constant sequences with the sup norm are continuous. This space is isomorphic to $c_{00}$ with the sup norm, so this shows that the property of all step-isometries being continuous is not preserved by isomorphism.

Finally the space $L_1(0,1)$ does not have \emph{any} extreme points,
so definitely no strongly extreme point. The following lemma shows
that the unit balls do form a subbase so, once again, we see that all
step-isometries are continuous on this space. However, in contrast to
spaces with a strongly extreme point, we need to consider the
intersection of more than two balls.

\begin{lemma}
  In $L_1(0,1)$ the open unit balls do form a subbase. However, for
  any two intersecting unit balls the diameter of the intersection is
  2.
\end{lemma}
\begin{proof}
  We start by proving the second part. By translating we may assume
  that one of the unit balls is centred at the origin and the other at
  some function $f$ in $L_1(0,1)$ with $\|f\|=2\lambda<2$. The result
  is trivial if $f=0$ so assume $f\not=0$. Pick $\alpha$ such that
  $\int_0^\alpha|f|=\int_\alpha^1|f|=\lambda$. Fix $\mu$ with $\lambda<\mu<1$.  Let
  $g_1=\frac{\mu}{\lambda} f \textbf{1}_{[0,\alpha]}$ and
  $g_2=\frac{\mu}{\lambda} f \textbf{1}_{[\alpha,1]}$. By the choice
  of $\lambda,\mu$ and $\alpha$ we see that $\|g_1\|=\|g_2\|=\mu$ so
  $g_1,g_2\in B(0,1)$.  Also, $\|g_1-f\|=(\mu-\lambda)+\lambda=\mu$ and
  similarly $\|g_2-f\|=\mu$ which shows that $g_1,g_2\in
  B(f,1)$. Finally we see that $\|g_1-g_2\|=2\mu$ which proves the
  result.

  For the first part, we show that, for any $\eps>0$ we can find 4
  unit balls whose common intersection contains 0, but is contained in
  $B(0,\eps)$. Let $\chi(\alpha,\beta)$ denote the function which is
  identically $\alpha$ on $[0,1/2]$ and identically $\beta$ on
  $(1/2,1]$.

  Fix $0<\delta<1$ and consider the four balls $B^\circ(\chi(2-\delta,0),1),
  B^\circ(\chi(-2+\delta,0),1),$ $ B^\circ (\chi(0, 2 -\delta), 1),
  B^\circ(\chi(0,-2+\delta),1)$. Obviously 0 is in the
  intersection. Suppose that $f$ is in the intersection. Since
  $\|\chi(2-\delta,0)-\chi(-2+\delta,0)\|=2-\delta$ we see that
  $f\textbf{1}_{[0,1/2]}$ must be at least $1-\delta/2$ from one of
  these functions. Since $f$ is within $1$ of that function, we see
  that $\|f\textbf{1}_{[1/2,1]}\|<\delta/2$. A similar argument applied
  to $\chi(0,2-\delta),\chi(0,-2+\delta)$ shows that
  $\|f\textbf{1}_{[0,1/2]}\|<\delta/2$ and hence that $\|f\|<\delta$
  and the result follows.
\end{proof}

As discussed above, one particular space not covered by this argument
is $c_0$ -- we will conclude this section by showing that even in that
case all step-isometries are continuous. However, before that we will
show that in spaces where the open unit balls form a sub-base, not
only are step-isometries on the whole space continuous, but
step-isometries on dense subsets are continuous. The example in the
previous section shows that this is not true in $c_0$.

\begin{lemma}
Let $X$ be a normed space where the open unit balls form a subbase for
the topology, let $S$ be a dense set in $X$ and let $T\colon
S\to S$ be a step-isometry. Then $T$ is continuous.
\end{lemma}
\noindent%
We remark that $S$ need not be countable, and indeed can be taken to
be the whole space.
\begin{proof}
  Fix $r>1$. First, observe that the set of all balls of radius $r$
  form a subbase. 
  Also we note that, since $r>1$, then for any $x,y\in
  X$ with $y\in B(x,r)$ there exists $z\in S$ such that $B(z,1)\subset
  B(x,r)$ and $y\in B(z,1)$. Indeed, the set $B(y,1)\cap B(x,r-1)$ is
  open and non-empty, and any point $z$ in this set satisfies these
  conditions.

  Now suppose $x\in S$ and $\eps>0$ are given. The ball $B(T(x),\eps)$
  is open so, since the balls of radius $r$ form a subbase, there exist
  $k\in \N$ and points $x\tog{1}$, $x\tog{2},\ldots,  x\tog{k}\in X$ such
  that \[\bigcap_{i=1}^kB(x\tog{i},r)\subset B(T(x),\eps)\qquad\text{and}\qquad
  T(x)\in \bigcap_{i=1}^kB(x\tog{i},r).\] By our observation we can
  find $z\tog{i}\in S$, for $1\le i\le k$, such that
  $B(z\tog{i},1)\subset B(x\tog{i},r)$ and $T(x)\in B(z\tog{i},1)$ which gives
\[\bigcap_{i=1}^kB(z\tog{i},1)\subset B(T(x),\eps)\qquad\text{and}\qquad
  T(x)\in \bigcap_{i=1}^kB(z\tog{i},1).\]

Then $\bigcap_{i=1}^kB(T^{-1}(z\tog{i}),1)$ is an open set containing
$x$ whose image is a subset of $B(T(x),\eps)$.
\end{proof}

The following corollary summarises what we have proved in this 
section.
\begin{corollary}\label{c:step-isom-cts}
    Suppose that $S$ is a dense set in a normed space $X$ and 
    let $T:S\to S$ be a step-isometry. Then $T$ is necessarily continuous if any of the following is true:
    \begin{enumerate}
    \item $X$ is reflexive.
    \item $X$ is the dual of an Asplund space. This includes all separable dual spaces. 
    \item $X=C(K)$ for some compact $K$.
    \item $X=C_b(K)$ for some space $K$ (bounded continuous functions on $K$).
    \item $X=L_1(0,1)$.
    \end{enumerate}    
\end{corollary}

\subsection*{Step isometries on $c_0$ are continuous}
In this section we give a complete description of step-isometries on
$c_0$ and, in particular, show they are all continuous. Apart from the
intrinsic interest, this means that we do not have any candidate
Banach spaces where step-isometries are not continuous. It also gives
a second proof of the fact that the (discontinuous) step-isometry on
$c_{00}$ constructed in Section~\ref{s:not-extend}  does not extend to a
step-isometry on $c_0$. More importantly though, it shows that the
continuity of step-isometries on a particular Banach space is not, by
itself, sufficient to show that the step-isometries on a dense subset
extend to step-isometries on the whole space -- thus, the complexity of
the argument in the next section, which shows that they do extend in
reflexive spaces, may be necessary.

We use two results from the theory of Banach spaces. The first is a
theorem of Omladi\v{c} and \u{S}emrl~\cite{MR1359952} (or
Gevirtz~\cite{MR0718987} building on Gruber~\cite{MR0511409} with a weaker constant) which implies that any
step-isometry is close to an isometry. Their precise result is:
\begin{theorem}[Omladi\v{c} and \u{S}emrl~\cite{MR1359952}]\label{t:omaldic-semrl}
  Let $X$ be a Banach space and $T:X\to X$ be a surjective map fixing zero
  with the property that, for some $\alpha$,
  $\big|\|T(x)-T(y)\|-\|x-y\|\big|\le \alpha$ for all $x,y\in X$. Then
  there exists a linear isometry $U$ of $X$ such that
  $\|U(x)-T(x)\|\le 2\alpha$.
\end{theorem}
We remark that the constant 2 in the conclusion is tight.
\begin{corollary}\label{c:omaldic-semrl}
Let $T$ be a step-isometry on a Banach space $X$ fixing zero. Then
there exists a linear isometry $U$ on $X$ such that
$\|T(x)-U(x)\|\le 2$ for all $x$.
\end{corollary}
\begin{proof}
  Since for any step-isometry we have $\lfloor\|x-y\|\rfloor
  =\lfloor\|T(x)-T(y)\|\rfloor$ we see that
  $\big|\|T(x)-T(y)\|-\|x-y\|\big|\le 1$, and Theorem~\ref{t:omaldic-semrl}
  can be applied with $\alpha=1$.
\end{proof}
The second result is a classical result describing the isometries of
$c_0$ (see e.g., Theorem~2.f.14. of~\cite{MR0500056}).
\begin{theorem}\label{t:c0-isom}
  Let $T$ be an isometry on $c_0$ fixing zero. Then there exists a
  permutation $\sigma$ of $\N$ and signs $\eps_i\in \{-1,+1\}$ such
  that $T\left((x_i)_{i=1}^\infty\right)=(\eps_ix_{\sigma(i)})_{i=1}^\infty$.
\end{theorem}

The next theorem shows that all step-isometries on $c_0$ have a
particular form. \u{S}emrl~\cite{MR4222374} gave an analogous
description for all spaces $\ell_p$ with $1<p\le \infty$ with
$p\not=2$. It is likely that  the description for $c_0$ follows
from his techniques but, as he does not state the result for $c_0$,
and the proof in this case is much simpler than the general case, we give it here.

\begin{theorem}\label{t:c0-step-isom}
  Let $T$ be a step-isometry on $c_0$ fixing zero. Then there exists a
  permutation $\sigma$ of $\N$ and step-isometries $f_i\colon \R\to\R$
  fixing zero such that $T\left((x_i)_{i=1}^\infty\right)=(f_i(x_{\sigma(i)}))_{i=1}^\infty$.
\end{theorem}
We remark that the converse is not true -- as we saw in
Section~\ref{s:not-extend}, maps defined in terms of step-isometries in
each coordinate may not even map $c_0$ to itself.
\begin{proof} 
By Corollary~\ref{c:omaldic-semrl} we know that every step-isometry $T$ on $X$ is close to an isometry
$U$. First suppose that $T$ is close to the identity: that is
$\|T(x)-x\|\le 2$ for all $x\in X$.

Fix $k\in \N$ and consider the map $f_k\colon\R\to\R$ defined by
$f_k(\lambda)=T(\lambda e_k)_k$. By our assumption on $T$ we know that
$\|T(\lambda e_k)-\lambda e_k\|\le 2$ so $|f_k(\lambda)-\lambda|\le2$
for all $\lambda\in \R$.

Now suppose $\lambda, \mu\in \R$ with
$\lfloor |\lambda-\mu|\rfloor =m\ge 5$. Since $T$ is a step-isometry
we know that $\lfloor\|T(\lambda e_k)-T(\mu e_k)\|\rfloor=m$ but we
also know that any coordinate of $T(\lambda e_k)$ or $T(\mu e_k)$
except the $k$th have absolute value at most $2$. Thus,
$\lfloor\|T(\lambda e_k)-T(\mu e_k)\|\rfloor=m$ implies that
$\lfloor|T(\lambda e_k)_k-T(\mu e_k)_k|\rfloor=m$, i.e. $\lfloor
|f_k(\lambda)-f_k(\mu)|\rfloor=m$. We have shown that each $f_k$ is a step-isometry for large distances (explicitly that it has
the step-isometry property for all $m\ge 5$), so by
Lemma~\ref{l:m-step-isometry-vector-space} is actually a
step-isometry.

By Lemma~\ref{l:integer-distances} we know that step-isometries
preserve integer distances: hence, for any integer $M$ and any
$\lambda$ we have $f_k(M+\lambda)=f_k(\lambda)\pm M$. Since
$|f_k(x)-x|\le 2$ for all $x\in\R$ we see that, for any $M\in\Z$ with
$|M|\ge 3$ we have $f_k(M+\lambda)=f_k(\lambda)+M$. 

We need to show that for any $x\in X$ we have $T(x)_i=f_i(x_i)$.  Pick
$M\in \N$ with $M>\|x\|+4$ and fix $k$. Then
$\|(M+x_k)e_k-x\|=M$. Again by Lemma~\ref{l:integer-distances} we have
$\|T((M+x_k)e_k)-T(x)\|=M$.

By definition $T((M+x_k)e_k)_k=f_k(M+x_k)$ and since
$\|T((M+x_k)e_k)-(M+x_k)e_k\|\le 2$ we also have
$|T((M+x_k)e_k)_i|\le2$ for all $i\not=k$. Similarly $\|T(x)-x\|\le 2$
so $|T(x)_i|\le |x_i|+2$ for all $i$. It follows that
$|T((M+x_k)e_k)_i-T(x)_i|\le |x_i|+4<M$ for all $i\not=k$.  Hence $
\|T((M+x_k)e_k)-T(x)\|=M$ implies that $|f_k(M+x_k)-T(x)_k|= M$.
Since $|M|\ge 3$, we have $f_k(M+x_k)=M+f_k(x_k)$ so this implies
$|M+f_k(x_k)-T(x)_k|= M$. Since $2M>|f_k(x_k)|+|T(x)_k|$ this implies
that $T(x)_k=f_k(x_k)$ as claimed.

We have shown that, in the case where $T$ is close to the identity it
is of the form
$T\left((x_i)_{i=1}^\infty\right)=(f_i(x_i))_{i=1}^\infty$  for
step-isometries $f_1,f_2\dots$.

Now consider a general step-isometry $T$ fixing zero. By
Corollary~\ref{c:omaldic-semrl} there exists an isometry $U$ of $c_0$
such that $\|T(x)-U(x)\|\le 2$ for all $x\in X$. Then $U^{-1}T$ is a
step-isometry and $\|U^{-1}T(x)-x\|=\|T(x)-U(x)\|\le2$ so the first
part of the proof applies to $U^{-1}T$: we have $U^{-1}T\left((x_i)_{i=1}^\infty\right)=(f_i(x_i))_{i=1}^\infty$  for
step-isometries $f_1,f_2\dots$.
 
Finally, since $U$ is an isometry we can apply
Theorem~\ref{t:c0-isom}: let $\eps_i$ and $\sigma$ be such that
$U\left((x_i)_{i=1}^\infty\right)=(\eps_ix_{\sigma(i)})_{i=1}^\infty$.
Hence
\[
T\left((x_i)_{i=1}^\infty\right) =U((f_i(x_i))_{i=1}^\infty) =(\eps_if_{\sigma(i)}(x_{\sigma(i)}))_{i=1}^\infty.
\]
Since $\eps_if_{\sigma(i)}$ is a step-isometry the result follows.
\end{proof}

\begin{corollary}
  Let $T$ be a step-isometry on $c_0$. Then $T$ is continuous.
\end{corollary}
\begin{proof}
  By translation, we may assume that $T$ fixes zero, and then it suffices to prove that $T$ is continuous at zero.
  
  Let $\sigma$ and $f_1,f_2,\ldots$ be as given by
  Theorem~\ref{t:c0-step-isom}. By permuting the coordinates, we may assume $\sigma$ is the identity. Suppose that $T$ is not continuous
  and let $\eps$ be such that for all $\delta>0$ there exist $x\in X$
  with $\|x\|<\delta $ and $\|T(x)\|> \eps$. Pick $x\ton$ such that
  $\|x\ton\|<1/n$ and $\|T(x\ton)\|>\eps$. There must exist $i=i(n)$
  such that $|T(x\ton)_i|=|f_i(x\ton_i)|>\eps$ and
  $|x\ton_i|<1/n$. Since each $f_i$ is continuous, by passing to a
  subsequence if needed, we may assume that the $i(n)$ is strictly
  increasing. Define
\[x=\sum_nx\ton_{i(n)} e_{i(n)}.
\]
Since the $i(n)$ are increasing and $|x\ton_i|<1/n\to 0$ we see that
$x\in c_0$. However, for any $n$,
$|T(x)_{i(n)}|=|f_{i(n)}(x\ton_{i(n)})|>\eps$ which shows that $T(x)$
is not in $c_0$ which is a contradiction.
\end{proof}

\section{Extending step-isometries on reflexive spaces}\label{s:extend-reflexive}
We have seen that in some Banach spaces (explicitly in $c_0$) it is
not always possible to extend step-isometries on a countable dense set
to step-isometries on the whole space. In this section we show that in any reflexive Banach space it is always possible to extend step-isometries on dense subsets to step-isometries on the whole space. Whilst countable dense sets in separable reflexive spaces are our main focus, since the proof works unchanged for dense sets, in any reflexive space we state the result for all reflexive spaces.

From the previous section we know that step-isometries on 
reflexive Banach spaces are necessarily continuous. This immediately
shows that any extension to a step-isometry is unique. However, we
do not see any way to use continuity to prove the existence of such an
extension. Indeed, the fact that, on $c_0$,  all step-isometries are
continuous, but it is not necessarily possible to extend a
step-isometry from a dense subset to the whole space, suggests that
there may not be a direct proof.

Instead, we work with the weak topology rather than the norm
topology. Since the space is reflexive, the unit ball is compact in the weak topology, and hence sequentially compact by the Eberlein-\v{S}mulian theorem (see e.g.,~\cite{MR0737004}).

\begin{theorem}\label{t:step-isos-extend}
  Let $T$ be a step-isometry on a dense set $S$ in a
  reflexive Banach space $X$. Then $T$ extends uniquely to a
  step-isometry $\bar T$ on the whole of $X$. 
\end{theorem}

As remarked above, since $S$ is dense and, by the results of
Section~\ref{s:continuous} any step-isometry on $X$ is continuous, the
uniqueness is immediate.

We are going to define the extension in terms of weak limits: for any
$x\in X$ take a sequence $x\ton\in S$ with $x\ton\wto x$ and define
$\bar T(x)$ to be the weak limit of the sequence $T(x\ton)$. Of
course, at the moment we do not know either that the sequence
$T(x\ton)$ weakly converges, or that different sequences weakly
converging to $x$ all give the same value for $\bar T(x)$.

We remark that we could define $\bar T$ in terms of sequences
$x\ton\to x$ in norm rather than weakly, but that does not seem to
simplify the argument. Also note that, the proof does make use of the
fact that $S$ is norm dense -- weak density would not be sufficient.

The overall structure of our argument in this section will be somewhat similar
to that of Section~4 of Balister, Bollob\'as, Gunderson, Leader and
Walters~\cite{MR3841851} but the proofs of many of the steps will be
very different as we will be working with weak limits not norm limits.
The key step is the following proposition which is analogous to
Lemma~18 of \cite{MR3841851}.

\begin{proposition}\label{p:seq-conv-banach}
  Let $T$ be a step-isometry on a dense set $\cds$ in a
   reflexive Banach space. Suppose that $(x\ton )_{n=1}^\infty$,
  $(y\ton )_{n=1}^\infty$ are sequences in $\cds$ weakly converging to
  $x$ and $y$ respectively, and that $T(x\ton )$, $T(y\ton )$ weakly
  converge to $x'$ and $y'$ respectively. Then, for any $m\in \N$ with
  $m\ge 3$ we have $\|x-y\|<m$ if and only if $\|x'-y'\|<m$.

  In particular, for any $m\in \N$ with $m\ge 3$ and any $y\in \cds$
  we have $\|x-y\|<m$ if and only if $\|x'-T(y)\|< m$.
\end{proposition}
This proposition would be straightforward to prove if all the
sequences converged in the norm topology (see Lemmas~17 and~18
of~\cite{MR3841851}) but we are only assuming convergence in the weak
topology. Since step-isometries are defined in terms of the norm we
need to relate the weak topology with the norm.

Suppose that we have a sequence $x\ton\wto x$ with $\|x\|>m$. It is
immediate that we can find a half-space (a set of the form
$H=\{z\in X:f(z)>\alpha\}$ for some $\alpha\in\R$ and a non-zero functional
$f\in X^*$), and $N\in\N$ such that $H$ is disjoint from $B(0,m)$ and
$x\ton\in H$ for all $n\ge N$. However, step-isometries need not
behave `nicely' with respect to half-spaces, so we would like to
approximate the half-space $H$ `near $x$' by a (large)
ball. Step-isometries do behave well with respect to balls -- at least
open balls with integer radius.

It would be natural to try to approximate the half-space by a ball
about a multiple of $x$, but that may not be possible -- the unit ball
may have a `corner' at $x$. For example even in two dimensions the
unit ball in $\ell_1^2$ has corners at the basis vectors, and we
cannot approximate the half space $\{(x,y)\in \ell_1^2,\ y>1\}$ near
$(0,1)$ by a ball about $(0,\lambda)$ for any $\lambda$.

However, at such a point we can pick a different half-space disjoint
from $B(0,1)$, and a ball approximating that half-space near
$(0,1)$. For example, the ball of radius $2k$ about $(k,k+1)$
approximates the half-space $\{(x,y):x+y>1\}$ near $(0,1)$.

Showing that, in any reflexive Banach space, we can
approximate half spaces by balls is the key result of this
section. Later, we will need to slightly strengthen this result
to insist that the approximating ball has integer radius and is
centred at a point of our dense set $S$. But, whilst that
requires a little care, it is essentially trivial, so we focus first on
proving the result for arbitrary balls.

\begin{lemma}\label{l:balls}
  Let $X$ be a  reflexive Banach space, $(x^{(n)})_{n=1}^\infty$ be
  a sequence in $X$ weakly converging to $x$, and $B$ be a closed
  ball not containing $x$. Then there exists an open ball $D$ with its
  closure $\bar D$ disjoint from $B$, and $N\in \N$ such that
  $x\ton\in D$ for all $n\ge N$.
\end{lemma}
\noindent%
We will use the \fr differentiability of the norm to prove this lemma. Recall, the norm
is \emph{\fd} at a point $x\in X$ if there exists $f_x\in X^*$ such that for all $\eps>0$
there exists $\delta>0$ such that
  \[
    \Big|\|x+y\|-\|x\|-f_x(y)\Big|\le \eps \|y\|,
  \]
  whenever $\|y\|\le\delta$. The functional $f_x$ is called the \emph{\fr
  derivative}. It is a norming functional for $x$: that is, $f_x$ has
  norm 1, and $f_x(x)=\|x\|$. Also if the norm is differentiable at
  $x$ then, for any $\lambda>0$, the norm is also differentiable at
  $\lambda x$ with $f_{\lambda x}=f_x$. 

  The example above of the unit ball in $\ell_1^2$ shows that the norm
  need not be \fd everywhere. However, Asplund~\cite{MR0231199} proved the following
  theorem.
\begin{theorem*}[Asplund's Theorem]
  Let $X$ be a Banach space such that every separable subspace has separable dual. Then, the set
  of points at which the norm is \fd is dense in $X$. 
\end{theorem*}
\noindent%
It is immediate that this theorem applies in any reflexive space.
\begin{proof}[Proof of Lemma~\ref{l:balls}]
  By a translation we may assume that $B$ is centred at $0$ and let
  $r$ be its radius. Further let $\hat x=x/\|x\|$ (since $x\not\in B$,
  $x\not=0$) and $\eps=(\|x\|-r)/4$. By Asplund's Theorem the norm is
  \fd at a dense set of points, so we may pick $\hat z$ be a vector of
  unit norm with
  $\|\hat z-\hat x\|\le \tfrac{\eps}{\|x\|}$ at which the norm is
  \fd, and let $f_{\hat z}$ be the derivative.

  For any vector $w$ we can write
  \[
  w=f_{\hat z}(w) \hat z   +(w-f_{\hat z}(w) \hat z),
  \]
  where the absolute value of the coefficient of $\hat z$ in the first term is at most
  $\|w\|$, and the second term is in $\ker f_{\hat z}$ and has norm at
  most $2\|w\|$.

  We apply this decomposition to each of the $x\ton$. Since $x\ton$ is weakly convergent, it is
  bounded, by $c$ say.  Hence, for each $n$ we can write
  \[
  x^{(n)}=\lambda_n\hat z+y^{(n)},
  \]
  where $\lambda_n=f_{\hat z}(x\ton)$ with $|\lambda_n|\le c$ and $y^{(n)}\in \ker f_{\hat z}$ with
  $\|y^{(n)}\|\le 2c$.

  Let $\eps'=\eps/2c$ and let $\delta$ be as in the definition of \fd
  for $\hat z$, $f_{\hat z}$ and $\eps'$ (that is, so
  $ \big|\|\hat z+y\|-\|\hat z\|-f_{\hat z}(y)\big|\le \eps' \|y\|$
  for any $y$ with $\|y\|\le\delta$).  Let $\alpha_0=2c/\delta$ and
  $z=(\alpha_0+c) \hat z$. Let $D=B^\circ(z,\|z\|-r-\eps)$. Obviously, the closure $\bar D$ is
  disjoint from $B$. It remains to show that $x\ton\in D$ for all
  sufficiently large $n$. We have
  \begin{align*}
    \|z-x^{(n)}\|&= \|z-\lambda_n\hat z-y^{(n)}\|\\
    &= \|(\alpha_0+c-\lambda_n)\hat z-y^{(n)}\|\\
    &= (\alpha_0+c-\lambda_n)\left\|\hat z-\frac{y^{(n)}}{(\alpha_0+c-\lambda_n)}\right\|\\
    &\le (\alpha_0+c-\lambda_n)\left(1+\eps'\left\|\frac{y^{(n)}}{(\alpha_0+c-\lambda_n)}\right\|\right)&    \qquad\text{since $\left\|\dfrac{y\ton}{(\alpha_0+c-\lambda_n)}\right\|\le \frac{2c}{\alpha_0}= \delta$}\\
    &= \|z\|-\lambda_n+\eps'\|y\ton\|\\
    &\le\|z\|-\lambda_n+\eps &\text{since $\|y\ton\|\le 2c$ }.
    \end{align*}
    Finally,
    \begin{align*}
      \lambda_n=f_{\hat z}(x\ton)&\to f_{\hat z}(x)\\
      &=f_{\hat z}(\|x\|\hat
      z)-f_{\hat z}(\|x\|\hat z-x)
      =\|x\|f_{\hat z}(\hat
      z)-\|x\|f_{\hat z}(\hat z-\hat x)\\
      &\ge \|x\|-\|x\|\|\hat z-\hat x\|
      \ge \|x\|-\eps,
    \end{align*}
    so there exists $N$ such that for all $n\ge N$, we have
    $\lambda_n> \|x\|-2\eps$ which implies that
    \[
      \|z-x\ton\|<\|z\|-\|x\|+3\eps=\|z\|-r-4\eps+3\eps=\|z\|-r-\eps,
    \] as required.
  \end{proof}
  As discussed before the previous lemma, we need a very slight
  strengthening of this result: we insist that the ball is centred at
  a point of our countable dense set, and that the ball has integer
  radius.

  \begin{lemma}\label{l:nice-our-case}
  Let $X$ be a reflexive Banach space,    $(x^{(n)})_{n=1}^\infty$ be a sequence in $X$ weakly converging to
  $x$, and $B$ be a closed ball not containing $x$. Further, let $S$ be a dense set in~$X$. Then there
  exists an open ball $D=B^\circ(s,m)$ for some $s\in S$ and $m\in \N$,
  disjoint from $B$, and $N\in \N$ such that $x\ton\in D$ for all $n\ge N$.
\end{lemma}
\begin{proof}
  Suppose that $B=B(w,r_1)$. By translating everything by $-w$ and
  replacing $S$ by $S-w$ we may assume $w=0$: i.e., $B=B(0,r_1)$. Let
  $D'$ be the disjoint ball as guaranteed by Lemma~\ref{l:balls}. We
  show that there is a ball of the required form containing $D'$.

  Suppose $D'=B^\circ(y,r_2)$ for some point $y\in X$ and $r_2>0$. Since $B$
  and $\bar D'$ are disjoint we have $r_1+r_2<\|y\|$. Let
  $\eps=\|y\|-r_1-r_2$. 

  Pick $\lambda>r_2$ of the form $\lambda=m-\eps/2$ for some
  $m\in \N$, and let $y'=y+(\lambda-r_2) y/\|y\|$. Then
  $\|y'\|=\|y\|+\lambda-r_2=r_1+m+\eps/2$. Moreover,
  $B(y',\lambda)\supset D'$.

  Now pick $s\in S$ with $\|s-y'\|<\eps/2$, which we can do because
  $S$ is dense.  Then $B^\circ(s,m)\supset B(y',\lambda)\supset D'$ and,
  since $\|s\|> \|y'\|-\eps/2=r_1+m$, the ball $B^\circ(s,m)$ is disjoint from
  $B$.
\end{proof}

Our next step is to use this to show that applying $T$ to a weakly
converging sequence does not increase distances by `too much'. This is
analogous to Lemma~17 of~\cite{MR3841851}, but whereas the proof there
(for sequences converging in norm) was trivial, here we need to use the
previous lemma as we only have weak convergence.
\begin{lemma}\label{l:seq-conv-banach}
  Let $X$ be a reflexive Banach space, $S$ a dense set in $X$. Suppose $T$ is a step-isometry on $\cds$, that $(x\ton )$ is a
  sequence in $\cds$ weakly converging to $x$, and that $T(x\ton )$ weakly converges to
  $x'$. Then, for any $y\in \cds$ and $m\in \N$ which satisfy $\|x-y\|<m$
  we have $\|x'-T(y)\|\le m$.
\end{lemma}
\begin{proof}
  Suppose for a contradiction that $\|x'-T(y)\|>m$. Let $B$ be the
  ball $B(T(y),m)$, so $x'\not \in B$. By Lemma~\ref{l:nice-our-case}
  we can find $s'\in S$, $m'\in \N$ and $N\in \N$ such that $B^\circ(s',m')$
  is disjoint from $B(T(y),m)$ and $T(x\ton)\in B^\circ(s',m')$ for
  all $n\ge N$. The disjointness condition shows that
  $\|T(y)-s'\|\ge m+m'$ and the second condition is exactly the
  condition $\|s'-T(x\ton)\|< m'$ for all $n\ge N$.

  Let $s=T^{-1}(s')$. Since $T$ is a step-isometry on $S$ we have
  $\|y-s\|\ge m'+m$, and  $\|s-x\ton\|<m'$ for all $n\ge N$.
  Since $x\ton\wto x$ the second of these bounds implies that
  $\|s-x\|\le m'$. This together with the first of these bounds and the
  triangle inequality then shows that $\|x-y\|\ge m$ which is the
  required contradiction.
\end{proof}

Next we prove Proposition~\ref{p:seq-conv-banach} which, since we now
have Lemma~\ref{l:seq-conv-banach}, is essentially the same as the
proof of Lemma~18 of~\cite{MR3841851}.
 \begin{proof}[Proof of Proposition~\ref{p:seq-conv-banach}]
   Suppose we have $x\ton,y\ton$, $x$, $y$, $x',y'$ and $m$ as in the
   proposition, and that $\|x-y\|<m$. Since $S$ is dense we can pick
   $s,t\in S$ with $\|x-s\|<1$, $\|s-t\|<m-2$ and $\|t-y\|<1$. By
   Lemma~\ref{l:seq-conv-banach} $\|x'-T(s)\|\le 1$ and
   $\|T(t)-y'\|\le 1$, and since $T$ is a step isometry on $S$ we have
   $\|T(s)-T(t)\|<m-2$. Hence, by the triangle inequality
   $\|x'-y'\|<m$.

   The reverse implication follows by applying the same argument to
   $T^{-1}$ which is also a step-isometry on $S$.

   Finally, if $y\in S$ then, by taking the sequence $y\ton$ to be the
   constant sequence $y\ton=y$, the last part of Proposition~\ref{p:seq-conv-banach} follows. 
 \end{proof}

 We can now prove that weak limits, if they exist, are unique. This is
 analogous to Lemma~19 of~\cite{MR3841851}.
 \begin{lemma}\label{l:injective}
  Let $X$ be a reflexive Banach space, $S$ be a dense set in $X$, and  $T$ be a step-isometry on $\cds$. Suppose that $(x\ton ),(y\ton )$ are
  two sequences in $\cds$ both weakly converging to some $x\in X$, and that $T(x\ton )$ and $T(y\ton )$
  weakly converge to $x'$ and $y'$, respectively. Then $x'=y'$.
\end{lemma}
\begin{proof}
  Suppose that $x'\not =y'$. Then the set
  \[
  \{v\in X: \|x'-v\|<3 \text{ and }\|y'-v\|>3\}
  \]
  is open and non-empty. Since $\cds$ is dense in $X$, there exists
  $z'\in \cds$ with $\|x'-z'\|<3$ and $\|y'-z'\|>3$. Let
  $z=T^{-1}(z')$. Then, the last part of Proposition~\ref{p:seq-conv-banach} applied to
  the sequence $(x\ton )$ and the element $z\in S$, and applied to the sequence $(y\ton )$ and
  $z$ implies $\|x-z\|<3$ and $\|x-z\|\ge 3$, which is a contradiction.
\end{proof}
Finally we use the sequential compactness of the unit ball in the
weak topology to show that whenever a sequence $x\ton$ weakly
converges then the sequence $T(x\ton)$ also weakly converges. This is
analogous to Lemma~20 of~\cite{MR3841851} but where the proof there used that the
unit ball in a finite-dimensional space is (sequentially) compact in the norm
topology, here we are using that the unit ball in a reflexive space is
sequentially compact in the weak topology.
\begin{lemma}\label{l:convergence-banach}
  Let $X$ be a reflexive Banach space, $S$ be a dense set in $X$, and  $T$ be a step-isometry on $\cds$. Suppose that $(x\ton )$ is a sequence in $S$ that weakly
  converges in $X$. Then $T(x\ton )$ is a weakly converging sequence.
\end{lemma}
\begin{proof}
  Since $(x\ton )$ is weakly convergent, it is bounded. Thus we can
  pick $c\in \N$ such that for all $n$ we have
  $\|x\ton-x^{(1)}\|<c$. Hence, since $T$ is a step-isometry, 
  $\|T(x\ton )-T(x^{(1)})\|<c$ for all $n$; in particular,
  $T(x\ton )$ is a bounded sequence. Thus, since the unit ball in $X$
  is sequentially compact in the weak topology, there is a subsequence
  $(x^{({n_i})})$ such that $T(x^{({n_i})})$ weakly converges to some
  value $x'$, say.

  Suppose that $T(x\ton )$ does not weakly converge to $x'$. Then
  there exists a subsequence which misses some weakly open set
  containing $x'$.  As above we can take a further subsequence which
  weakly converges to $x''$ say. Since this subsequence misses a
  weakly open set containing $x'$, we must have $x''\not=x'$. But this
  contradicts Lemma~\ref{l:injective}.
\end{proof}
Finally, we put these results together to define the extension
$\bar T$.
\begin{proof}[Proof of Theorem~\ref{t:step-isos-extend}]
  We define $\bar T$ in the obvious way: for any $x\in X$ pick a
  sequence $x\ton\wto x$ and define $\bar T(x)$ to be the weak limit
  of $T(x\ton)$. By Lemma~\ref{l:convergence-banach} the sequence
  $T(x\ton)$ does weakly converge, and by Lemma~\ref{l:injective} the
  resulting function is well-defined (any sequence weakly converging
  to $x$ gives the same limit).

  By Proposition~\ref{p:seq-conv-banach} for any $m\in \N$ with $m\ge 3$ we have
  $\|x-y\|<m$ if and only if $\|\bar T(x)-\bar T(y)\|<m$, which by
  Lemma~\ref{l:m-step-isometry-vector-space} shows that $\bar T$ is a step-isometry.
\end{proof}

\section{Conditions that ensure step-isometries are isometries}\label{s:step=isom}
In this section we prove Theorem~\ref{t:strict-convex-step-iso-is-iso}.
The following lemma is implicit in sections~5 and~6 of~\cite{MR3841851} .
\begin{lemma}\label{l:bbglw-s5s6}\cite{MR3841851}
  Suppose that $X$ is a normed space and that $T$ is a 
  step-isometry on $X$ that fixes $0$ and preserves integer
  distances. Let \[\Lambda=\ \left\{\sum_i\lambda_i x_i: \lambda_i\in \Z \text{\ and\ }
  x_i\in \ext(B)\right\}\] where $Ext(B)$ denotes the set of extreme points of $B$. Then $T|_{\Lambda}$ maps $\Lambda$ onto itself and is an isometry.
  Further, if $T$ is continuous then $T|_{\bar \Lambda}$ maps $\bar\Lambda$ onto itself and is an isometry.
\end{lemma}
Note that, by Lemma~\ref{l:integer-distances} the condition that $T$
preserves integer distances is automatically satisfied. We have
included it in the statement as it is included in the (implicit)
statement of Lemma~\ref{l:bbglw-s5s6} in~\cite{MR3841851}. We also remark that whilst the continuity of $T$ is assumed throughout Sections 5 and 6 of~\cite{MR3841851}, it is only used to show that $T|_{\bar\Lambda}$ being an isometry follows from $T|_{\Lambda}$ being an isometry.

Finally note that $\Lambda$ is not typically discrete and, indeed, in all our applications $\Lambda$ will be dense in the whole of $X$ and, in many cases, will actually be the whole of $X$. The following lemma gives some sufficient conditions for these to occur. 
 \begin{lemma}\label{l:lambda-dense}
   Let $X$ be a Banach space of dimension at least~2 with the property
   that the extreme points of its unit ball are dense in the unit
   sphere. Then the set $\Lambda$ is dense
   in $X$. Further, if $X$ is strictly convex then $\Lambda$ is the whole of $X$.
 \end{lemma}

 \begin{proof}
   Since the extreme points are dense in the unit sphere, it suffices
   to show that any point in the unit ball can be written as the sum
   of two unit vectors. This is trivial for the zero vector, so let
   $x\not =0$ be a vector with $\|x\|<1$.

  The vector $x/\|x\|$ is in the unit sphere and has distance less
  than one from $x$; in contrast the vector $-x/\|x\|$ is in the unit
  sphere and has distance greater than one from $x$. Hence, since the
  unit sphere is connected (since the dimension is at least $2$), there
  is a point $y$ in the unit sphere at distance exactly $1$ from $x$. We
  see that $x=y+(x-y)$ is the required sum.
  When $X$ is strictly
    convex, all points of the unit sphere are extreme points, hence the second statement follows.
\end{proof}

\begin{proof}[Proof of Theorem~\ref{t:strict-convex-step-iso-is-iso}]
    This is immediate from the results we have.
    Theorem~\ref{t:step-isos-extend} shows that $T$ extends to a
    step-isometry on the whole of $X$. Since $X$ is strictly
    convex, all points of the unit sphere are extreme points, so
    Lemma~\ref{l:lambda-dense} combined with Lemma~\ref{l:bbglw-s5s6}
    (using the fact that Lemma~\ref{l:integer-distances} shows that
    $T$ preserves integer distances) show that $T$ is an isometry on
    $X$. Since $T$ maps $S$ to itself, it follows that $T$ is an
    isometry on $S$.
\end{proof}
We remark that this proof actually shows that Theorem~\ref{t:step-isos-extend} holds under the weaker hypothesis that the extreme points of the unit ball are dense in the unit sphere, or the even weaker condition that the lattice $\Lambda$ generated by the extreme points of the unit ball is dense in the whole space. However, for simplicity we have stated Theorem~\ref{t:step-isos-extend} just for strictly convex spaces.

\section{A Rado set in any space with a 1-symmetric
  basis.}\label{s:ellp-rado-set} In this section we prove
Theorem~\ref{t:exist-rado}. Recall that a basis $(e_i)$ in a Banach
space is \emph{$1$-symmetric} if, for any permutation $\pi$ of $\N$ and any
sequence $a_i$ of real numbers such that $\sum_ia_ie_i$ converges, the
sum $\sum_ia_{\pi(i)}e_i$ also converges and
$\|\sum_ia_ie_i\|=\|\sum_ia_{\pi(i)}e_i\|$. Trivially, in any $\ell_p$
space, the standard basis is a $1$-symmetric basis.

We remark that we first constructed a Rado set in $\ell_2$ and then
observed that the only property of a Hilbert space we were using was
that it had a $1$-symmetric basis. A reader unfamiliar with symmetric
bases will lose very little by restricting everything to the case of
the standard basis in $\ell_2$.

Let us now fix a Banach space $X$ with a $1$-symmetric basis $(e_i)$. We construct a Rado set
$S$ in $X$.

\subsection*{Construction}
We are going to construct a
countable dense set $S=\{s^{(1)},s^{(2)},\dots\}$ such that $\supp(s^{(i)})\subseteq [i]$
and $s^{(i)}_i\not =0$. In particular the $s^{(i)}$ will be linearly
independent. We start by letting
\[Q_n=\{\text{$x\in \Q:$  $|x|\le n$, $x=a/b$, $a,b\in \N$,
$b\le n$}\},\] and $Q_n'=Q_n\setminus\{0\}$.  Let
\[U_n=\left\{x\in X: \text{$x_i\in Q_n$ for $1\le i< n$ and $x_i=0$
      for $i\ge n$} \right\}.
\]
We are going to define sets $S_n$ inductively based on all of the
$S_i$ for $i<n$ together with $U_n$. We will then define $S=\bigcup_n
S_n$. Let $S_1=\{e_1\}$. Suppose that, we have defined $S_1,S_2,\dots,
S_{n-1}$, and that $\bigcup_{i\le n-1}S_i=\{s^{(1)},s^{(2)},\ldots ,
s^{(N)}\}$. Note that, provided we add at least one point at each
stage, we will have $N\ge n-1$ and in particular $U_n\subset
\left<e_1,e_2,\dots,e_N\right>$.  Enumerate the pairs of $U_n\times
Q'_n$ as $(u_1,q_1),(u_2,q_2),\dots, (u_m,q_m)$. Let
$s^{(N+i)}=u_i+q_ie_{N+i}$ and $S_n=\{s^{(N+i)}: 1\le i\le m\}$.

Obviously $S$ is countable, $\supp(s^{(i)})\subseteq [i]$, and
$s^{(i)}_i\not = 0$. Finally, $S$ is dense in $c_{00}$: indeed, for any
point $x\in c_{00}$ with rational coordinates and any $n$ we can find $k$ such that the point
$x'=x+e_k/n\in S$. Since $c_{00}$ is dense in $\ell_p$ this implies
that $S$ is dense in $\ell_p$.

\subsection*{The set is Rado}
Now we have to show that this set $S$ is a Rado set. Suppose that we
have instances $G$ and $G'$ of $\mathcal{G}(S,p)$. As usual we use a
back-and-forth argument.

At each step of the argument we will have a partial permutation
$\sigma$ on $\N$: that is $\sigma$ is a bijection between two finite
subsets $I,J\subset \N$. Note these will \emph{not} typically be
initial segments of $\N$. We will insist that $\sigma$ has the
following properties:
\begin{enumerate}
  \item The set $S_I=\{s\in S: \supp(s)\subset I\}$ spans $\R^I$ (so
    is necessarily a basis) and similarly for $J$.\label{con:span}
  \item The induced map $\hat \sigma:\left<e_i:i\in I\right>\to
    \left<e_j:j\in J\right>$ defined by $\hat\sigma(e_i)=e_{\sigma(i)}$
    maps $S_I$ to~$S_J$.\label{con:consistent}
  \item The map $\hat\sigma:S_I\to S_J$ is an isomorphism from $G|_{S_I}$ to
  $G'|_{S_{J}}$.\label{con:graph-iso} 
\end{enumerate}
Since the basis is $1$-symmetric, the map $\hat\sigma$ is an isometry,
so at each stage $S_I$ and $S_J$ are isometric.

Let $i=\min(\N\setminus I)$. We want to extend the partial permutation
to $I'=I\cup\{i\}$. Observe that Property~\ref{con:span} does hold for
$S_{I'}$ since $\supp(s^{(i)})\subseteq [i]\subseteq I'$ implies that
$s^{(i)}\in S_{I'}$ and $s^{(i)}_i\not=0$ implies that $s^{(i)}$ is
linearly independent of $S_I$.

As $\sigma$ is already defined on $I$, in order to define it on $I'$
we just have to decide where to send $i$. We need to find $j\not \in
J$ such that extending $\sigma$ by defining $\sigma(i)=j$ maintains
the three conditions above. Let $J'=J\cup \{j\}$. To ensure
Property~\ref{con:consistent} holds we need $\hat\sigma(s^{(i)})\in
S_{J'}$. But $\hat\sigma(s^{(i)})=\hat\sigma(s^{(i)}|_I)+s^{(i)}_ie_j$.
Since $\hat\sigma(s^{(i)}|_I)\in U_n$ for some $n$, by construction
there are infinitely many values of $j$ for which $\hat\sigma(s^{(i)})\in
S_{J'}$, so almost surely one
of them will also satisfy Property~\ref{con:graph-iso}, which
completes this step. This choice of $J'$ will automatically satisfy
Property~\ref{con:span}.

We repeat this in the standard back-and-forth method, and we see that
$\sigma$ extends to be a bijection on the whole of $\N$ inducing the
required graph isomorphism.

\section{A Banach space where all sets are strongly non-Rado}\label{s:no-rado}

In this section we put together our earlier results with general
results from Banach space theory to prove
Theorem~\ref{t:exist-no-rado} showing that there are infinite
dimensional Banach spaces where \emph{all} countable dense sets are strongly
non-Rado.

Suppose that we can find a space $X$ which is separable, reflexive, strictly
convex and has no non-trivial isometries (an isometry is trivial if it
is $\pm I$). Then, for any countable dense set $S$ of $X$,
Theorem~\ref{t:strict-convex-step-iso-is-iso} combined with
Lemma~\ref{l:countable=not-isomorphic} shows that $S$ is strongly
non-Rado.

Many spaces with no non-trivial isometries have been constructed but
we have not been able to find an example which has all the properties
we need.  However, it is simple to modify a construction of Davis~\cite{MR0298393}
to give a space, which is separable, reflexive and has no non-trivial
isometries and, whilst it is not strictly convex, it is `close enough' to
strictly convex that Lemma~\ref{l:bbglw-s5s6}
shows that any step-isometry is an isometry.

\begin{lemma}\label{l:exist-X}There exists a Banach space $X$ that is separable,
  reflexive, has no non-trivial isometries, and has the property that
  the lattice $\Lambda$ generated by the extreme points of its unit
  ball is the whole space.
\end{lemma}
\begin{proof}
  Following Davis we renorm $\ell_2$ such that the only isometries are
  $\pm I$. View $\ell_2$ as having basis $e_0,e_1,\dots$.
  Let 
  \[F=\left\{e_0+\frac{1}{2n}e_n:n\ge 1\right\}\cup \left\{e_0-\frac{1}{2n+1}e_n:n\ge
    1\right\}.\] Define the unit ball $B'$ to be the closed convex hull of
  the set 
  \[G=\left\{x\in \ell_2:\|x\|=1\text{ and }|x_0|\le \frac{1}{10}\right\}\cup F \cup
    -F\] and let $\|\cdot\|'$ denote the corresponding norm. Let $X$ be the space $\ell_2$ with norm $\|\cdot\|'$.

  First we do the trivial check that this is an equivalent norm. Since
  $\|x\|\le 3/2$ for all $x\in F$ we see that $B'\subseteq \frac32 B$.
  Conversely suppose $x\in B$. Then
  $x=\alpha e_0+y$ with 
  $y\cdot e_0=0$ and
  $\alpha^2+\|y\|^2\le 1$. Since $e_0\in B'$ and $y\in B'$
  we see that $x=\alpha e_0+y\in 2B'$. Therefore $B\subseteq 2B'$ completing the proof of the equivalence of  $\|\cdot\|$ and $\|\cdot\|'$. This shows that  $X$ is separable and reflexive.

  Next, let $T$ be any surjective isometry fixing $0$.  We need to
  show that $T=\pm I$. By the Mazur-Ulam
  Theorem $T$ must be linear.  Since $T$ fixes $0$ it must map the
  unit ball to itself, and since it is linear, it must map the extreme
  points of the unit ball $B'$ to themselves. It is easy to see that the
  set of extreme points of the unit ball is exactly $G$.

  The points $\pm(e_0+\frac12e_1)$ are the only points which have no
  extreme points within distance $\frac 12$. Therefore this pair must
  map to this pair. By applying the isometry $-I$ we may assume they
  map to themselves.

  Now of the remaining extreme points $\pm(e_0-\frac13e_1)$ are the
  only ones which have no extreme points within distance a $\frac13$
  so they map as a pair to the same pair.  Since $e_0+\frac12e_1$ is
  fixed we see that $T$ must be the identity on these two
  points. Repeating this shows that $T$ is the identity on all points
  in $F$.  Since $F$ spans all of $X$ this implies the map is the
  identity.

  Finally we need to show that $\Lambda$ is the whole space. Since
  $\{x\in X:\|x\|=1\text{ and }x_0=0\}$ are all extreme points, Lemma~\ref{l:lambda-dense} applied to the subspace $\{x\in X:x_0=0\}$ shows that $\Lambda\supset \{x\in X:x_0=0\}$. Now given any
  $y\in X$ we can find $z\in\Lambda$ with $z_0=y_0$. But then
  $y-z\in \{x\in X:x_0=0\}$ so $y\in z+\Lambda=\Lambda$.
\end{proof}
\begin{proof}[Proof of Theorem~\ref{t:exist-no-rado}] 
Suppose that $S$ is a countable dense subset of the normed space $X$ given by Lemma~\ref{l:exist-X} and that $T$ is a step-isometry on $S$.  Since $X$ is a separable reflexive space, Theorem~\ref{t:step-isos-extend} shows that $T$ extends to a step-isometry $\bar T$ on the whole of $X$. By construction (Lemma~\ref{l:exist-X}) $\Lambda$ is the whole space $X$, so by Lemma~\ref{l:bbglw-s5s6}, $\bar T$ is an isometry on the whole of $X$. Hence, by Lemma~\ref{l:exist-X}, $\bar T-\bar T(0)=\pm I$. Since $\bar T$ must map $S$ to itself, there are only countably many possiblities for $\bar T(0)$, which in turn implies that there are countably many step-isometries of $S$. Finally, Lemma~\ref{l:countable=not-isomorphic} shows that $S$ is strongly non-Rado.
\end{proof}

\section{Open Questions}
Whilst this paper has proved that a large class of reflexive spaces have
the property that most countable dense sets are strongly non-Rado, this is far from
a full classification. Given the complexity even in finite dimensions
we expect a full classification to be complicated. As the key step in our proofs is showing that step-isometries are isometries, we phrase the question in
that form.
\begin{question}
  For which reflexive spaces $X$ are step-isometries
  necessarily isometries?
\end{question}

The main open question, however, is whether typical sets are Rado or
non-Rado in non-reflexive Banach spaces. As mentioned in the
introduction, some specific cases are known ($c_0$, $c$ and
$C(0,1)$). But it is open for all other non-reflexive spaces -- in
particular the following very concrete case is open.

\begin{question}
  Is a typical countable dense set in $\ell_1$ Rado, or strongly non-Rado?
\end{question}

Some of the techniques of this paper may be applicable by moving to
the weak-* topology and working on the double dual $X^{**}$. However,
since the dual of $\ell_1$ is not separable, the weak-* topology on
its double dual is neither metrisable nor sequentially compact, and
moreover, $\ell_1$ is not an Asplund space.

Given this, the case of a non-reflexive space with a separable dual may
be more tractable. Indeed, we conjecture that such spaces behave very much
like the reflexive spaces.
\begin{conjecture}
  Let $X$ be a strictly convex Banach space with separable dual. Then
  `typical' countable dense sets in $X$ are strongly non-Rado. 
\end{conjecture}
As we have seen (Lemmas~\ref{l:integer-distances},~\ref{l:bbglw-s5s6} and~\ref{l:lambda-dense}) step-isometries on the whole of any strictly convex space are
necessarily isometries, the above conjecture would be implied by the
following conjecture given purely in terms of Banach spaces.
\begin{conjecture}
  Let $X$ be a strictly convex Banach space with separable dual and
  $T$ a step-isometry on a countable dense set $S$ in $X$. Then $T$
  can be extended to a step-isometry on the whole of~$X$.
\end{conjecture}
We have also seen that step-isometries are continuous on many Banach
Spaces. But are there any cases when they are not? We do not know, so
we leave this as a question.
\begin{question}
  Suppose that $X$ is a Banach space and that $T$ is a
  step-isometry on $X$. Must $T$ be continuous?
\end{question}
\bibliographystyle{abbrv} 
\bibliography{dense-geometric-mybib}

\end{document}